\documentclass[11pt, a4paper]{article}
\usepackage{}
\usepackage{amsthm}
\usepackage{mathrsfs}
\usepackage{amsmath,amssymb,latexsym,color}
\usepackage[colorlinks,
linkcolor=blue,
anchorcolor=green,
citecolor=magenta
]{hyperref}
\usepackage{graphicx}
\usepackage{tikz}
\usetikzlibrary{calc}
\usepackage{CJK}

\newcommand{\qbinom}[2]{{#1\brack #2}}

\usepackage{authblk}

\long\def\delete#1{}

\usepackage{color}

\definecolor{Blue}{rgb}{0,0,1}
\definecolor{Red}{rgb}{1,0,0}
\definecolor{DarkGreen}{rgb}{0,0.6,0}
\definecolor{DarkYellow}{rgb}{1,1,0.2}
\definecolor{DarkPurple}{rgb}{.6,0,1}

\usepackage{xcolor}
\usepackage[normalem]{ulem}

\usepackage{enumerate}
\usepackage{cleveref}
\crefformat{section}{\S#2#1#3}
\crefformat{subsection}{\S#2#1#3}
\crefformat{subsubsection}{\S#2#1#3}
\crefrangeformat{section}{\S\S#3#1#4 to~#5#2#6}
\crefmultiformat{section}{\S\S#2#1#3}{ and~#2#1#3}{, #2#1#3}{ and~#2#1#3}
\def\ma{\mathcal{A}}
\def\mb{\mathcal{B}}

\def\md{\mathcal{D}}

\def\mf{\mathcal{F}}
\def\mg{\mathcal{G}}
\def\mh{\mathcal{H}}
\def\mi{\mathcal{I}}

\def\bs{\setminus}

\def\ge{\geqslant}
\def\le{\leqslant}
\def\b{\brack}

\numberwithin{equation}{section}

\newtheorem{thm}{Theorem}[section]
\newtheorem{lem}[thm]{Lemma}

\newtheorem{ex}[thm]{Example}

\newtheorem{cl}{Claim}

\begin{document}
	
	\setcounter{page}{1}
	\renewcommand{\thefootnote}{}
	\newcommand{\remark}{\vspace{2ex}\noindent{\bf Remark.\quad}}
	\renewcommand{\abovewithdelims}[2]{%
		\genfrac{[}{]}{0pt}{}{#1}{#2}}

	
	\def\qed{\hfill$\Box$\vspace{11pt}}
	
	\title {\bf $s$-almost cross-$t$-intersecting families for vector spaces}

	\author{Dehai Liu\thanks{E-mail: \texttt{liudehai@mail.bnu.edu.cn}}\   \textsuperscript{a}}
	\author{Jinhua Wang\thanks{Corresponding author. E-mail: \texttt{202321130084@mail.bnu.edu.cn}}\   \textsuperscript{b}}
	\author{Tian Yao\thanks{ E-mail: \texttt{tyao@hist.edu.cn}}\ \textsuperscript{c}}

	\affil{ \textsuperscript{a} Laboratory of Mathematics and Complex Systems (Ministry of Education), School of
		Mathematical Sciences, Beijing Normal University, Beijing 100875, China}
	
	\affil{ \textsuperscript{b} School of
		Mathematical Sciences, Beijing Normal University, Beijing 100875, China}

	\affil{ \textsuperscript{c} School of Mathematical Sciences, Henan Institute of Science and Technology, Xinxiang 453003, China}

	\date{}
	
	\openup 0.5\jot
	\maketitle

	\begin{abstract}
		
		Let $V$ be an $n$-dimensional vector space over the finite field $\mathbb{F} _{q} $, and ${V\brack k}$ denote the family of all $k$-dimensional subspaces of $V$. The families $\mathcal{F},\mathcal{G}\subseteq {V\brack k}$ are said to be cross-$t$-intersecting if $\dim(F\cap G)\ge t$ for all $F\in \mathcal{F}, G\in \mathcal{G}$. Two families $\mathcal{F}$ and $\mathcal{G}$ are called $s$-almost cross-$t$-intersecting if each member of $\mathcal{F}$ (resp. $\mathcal{G}$) is $t$-disjoint with at most $s$ members of  $\mathcal{G}$ (resp. $\mathcal{F}$).
		In this paper, we discribe the structure of $s$-almost cross-$t$-intersecting families with maximum product of their sizes. In addition, we prove a stability result.
	
		\vspace{2mm}
		
		\
		
		
	\end{abstract}
	
	\section{Introduction}

	Let $n,k$ and $t$ be positive integers with $n\ge k\ge t$. Write $[n]=\left \{ 1, 2,...,n \right \}$. For a set $X$, denote the family of all $k$-subsets of $X$ by $\binom{X}{k} $ and let $\mathcal{F} $ be a family consisting of $k$-subsets of an $n$-set. 
	It is called  \textit{$t$-intersecting} if $\left| F\cap F^{\prime}\right|\geq t$ for any $F, F^{\prime}\in \mf$. 
	The famous Erd\H{o}s-Ko-Rado theorem \cite{MR0140419, MR0771733} gives the structure of maximum-sized $t$-intersecting families.  
	We refer readers to  \cite{ 2406295, 2407161, MR0519277, 2407162, MR0771733} for extensive  results on $t$-intersecting families.
	
	We say that $\mf \subseteq \binom{[n]}{k} $ is  \textit{$s$-almost $t$-intersecting} if $\left | \left \{ F\in \mathcal{F}:\left | F\cap F'  \right |<t   \right \}  \right | \le s$ for any $F^{\prime}\in\mf$. Clearly, a $t$-intersecting family is $s$-almost $t$-intersecting. 
	Gerbner et al. \cite{2406095}  proved that,  when  $n$ is sufficiently large, each member of a maximum-sized $s$-almost $1$-intersecting family contains a fixed element. 
	In \cite{2406096}, Frankl and Kupavskii characterized the structure of maximum $1$-almost $1$-intersecting families under the condition that they are not $1$-intersecting. The $s$-almost $t$-intersecting families were also investigted in \cite{finite sets}.

Two families $\mathcal{F}, \mathcal{G} \subseteq \binom{[n]}{k} $ are said to be \textit{cross $t$-intersecting} if $|F\cap G|\ge t$ for any $F\in \mathcal{F}, G\in \mathcal{G}$. There are numerous papers on studying the maximum sum or the maximum product of size of cross $t$-intersecting families under different conditions. We refer the readers to \cite{chengji1, chengji2, he1, he2, 2407162}.

One may further weaken the cross-$t$-intersecting condition locally. Two families $\mathcal{F}, \mathcal{G}\subseteq \binom{[n]}{k}$ are called $s$-almost cross-$t$-intersecting if $\left | \left \{ G\in \mathcal{G}:\left | G\cap F \right |<t   \right \}  \right | \le s$ for any $F\in \mathcal{F}$ and $\left | \left \{ F\in \mathcal{F}:\left | F\cap G \right |<t   \right \}  \right | \le s$ for any $G\in \mathcal{G}$. In \cite{jihu1}, Gerbner et al. pointed out the maximum value of $|\mathcal{F}||\mathcal{G}|$ for $t=1$. In \cite{almost cross}, the authors characterize the families $\mathcal{F}$ and $\mathcal{G}$ with the maximum $|\mathcal{F}||\mathcal{G}|$ for $t \ge  1$ and study $s$-almost cross-$t$-intersecting families which are not cross-$t$-intersecting.

	Intersection problems are  studied on some other mathematical objects, for example, vector spaces. Let $V$ be an $n$-dimensional vector space over the finite field $\mathbb{F}_{q}$. Write  the family of all
	$k$-subspaces of $V$ as ${V\b k}$. Recall that for any positive integers $a$ and $b$ the \textit{ Gaussian binomial coefficient} is defined by
	$${a\brack b}=\prod_{0\le i< b} \frac{q^{a-i}-1 }{q^{b-i}-1}. $$
	In addition, we set ${a\brack 0}=1$, and ${a\brack b}=0$ if $b$ is a negative integer. The size of ${V\brack k}$ is equal to ${n\brack k}$.
	
	For a positive integer $t$, a family $\mf\subseteq {V\b k}$ is said to be \textit{$t$-intersecting} if $\dim(F\cap F^{\prime})\geq t$ for any $F,F^{\prime}\in \mf$. The structure of $t$-intersecting families of ${V\brack k}$ with maximum size had been completely determined \cite{MR0867648, MR0382015, MR2231096} which are known as the Erd\H{o}s-Ko-Rado theorem for vector spaces.
	Some classical results were proved for vector spaces \cite{ 2406292, MR0721612, 2406294,2405263}.

	Define $\mathcal{D}_{\mathcal{F} } (F;t)  =\left\{F\in\mf: \dim(F\cap F^{\prime})<t\right\}$. As a generalization of $t$-intersecting families for vector spaces, a family $\mf\subseteq {V\b k }$ is called \textit{$s$-almost $t$-intersecting} if $\left|\mathcal{D}_{\mathcal{F} } (F;t) \right|\leq s$ for any $F^{\prime}\in\mf$.  Shan and  Zhou \cite{2406097} characterized $1$-almost $1$-intersecting families with the maximum size. Ji et al. \cite{Ji}  provides the structure of maximum-sized $s$-almost $t$-intersecting families for general $s$ and $t$.
	
We say that two families $\mathcal{F} ,\mathcal{G} \subseteq \qbinom{V}{k}$ are \textit{cross-$t$-intersecting} if $\dim ( F\cap G )\ge t $ for any $F\in  \mathcal{F}$ and $G\in \mathcal{G} $. We refer the readers to  \cite{交叉t相交向量4, 交叉t相交向量3, 交叉t相交向量2, 交叉t相交向量1} for some  results on cross-$t$-intersecting families.
	
	Similarly, the local weakening of the cross-$t$-intersecting condition can be further pursued. Write $\mathcal{D} _{\mathcal{G} } (F;t)=\left \{ G\in\mathcal{G}:\dim(G\cap F)< t  \right \} $. Two families $\mathcal{F} ,\mathcal{G} \subseteq \qbinom{V}{k}$ are called \textit{$s$-almost cross-$t$-intersecting} if $\left | \mathcal{D} _{\mathcal{G} } (F;t)\right | \le s$ for any $F\in  \mathcal{F}$ and $\left | \mathcal{D} _{\mathcal{F} } (G;t)\right | \le s$  for any $G\in \mathcal{G} $.


	\begin{thm}\label{dingli1}
		Let $n$, $k$, $t$ and $s$ be positive integers with $k\geq t+1$ and  $n\geq 2k+2t+1+\log_{q} 7s$. If $\mathcal{F},\mathcal{G}\subseteq {V\brack k}$ are $s$-almost cross-$t$-intersecting families, then 
		$$|\mathcal{F}||\mathcal{G}|\le {n-t\brack k-t}^{2}. $$  Equality holds if and only if $\mathcal{F}=\mathcal{G}=\left \{ H\in{V\brack k}: E\subseteq H  \right \} $ for some $E\in {V\brack t}$.
	\end{thm}
	
In extremal combinatorics, after identifying the sets that satisfy the given conditions, one may study the stability results. Here, we consider the $s$-almost cross-$t$-intersecting families $\mathcal{F},\mathcal{G}\subseteq {V\brack k}$ with $\dim \left ( \bigcap_{H\in \mathcal{F}\cup \mathcal{G} } H   \right ) <t$.
To present our another  result,  we introduce some families and examples. Suppose that $E,X$ and $L$ are subpaces of $V$ with $E\subseteq X$. Write

    \begin{equation*}
    	\begin{aligned}
    		\mh_{1}(X,E;k)&= \left\{F\in {X\brack k}: E\subseteq F\right\},\\
    		\mh_{2}(X,E;k)&=\left\{F\in{V\b k}: F\cap X=E\right\},\\
    		\mathcal{M}(L;k,t) & =\left \{ F\in {V\brack k}:\dim(F\cap L)\ge t \right \} .
    	\end{aligned}
    \end{equation*}

	
		\begin{ex}\label{lizi} 
		Let $n$, $k$,  $t$ and $s$ be positive integers with  $q^{(k-t)(k-t+1)} \qbinom{n-k-1}{k-t}>s$, and  $X\in  \qbinom{V}{k+1}$ and $E\in  \qbinom{X}{t}$. The families 
		$$\mh_{1}(V,E;k)\setminus \mathcal{A}~~ and~~\mathcal{H} _{1} (V,E;k)\cup \mathcal{B}, $$
		where $\ma$ is a $\left ( q^{(k-t)(k-t+1)} \qbinom{n-k-1}{k-t}-s \right ) $-subset of $\mathcal{H} _{2} (V,E;k)$ and $\mb$ is a $\min\left\{s, q^{(k-t+1)} \qbinom{t}{1}\right\}$-subset of ${X\brack k}\setminus\mh_{1}(V,E;k)$, are $s$-almost cross-$t$-intersecting but not cross-$t$-intersecting.
		
	\end{ex}
	
	The product of sizes of families in Example \ref{lizi} is
	\begin{align}\label{g1}
		\left ( {n-t\brack k-t}-q^{(k-t)(k+1-t)}{n-k-1\brack k-t}+s  \right ) \left ( {n-t\brack k-t}+\min\left\{s, q^{(k-t+1)} \qbinom{t}{1}\right\} \right ) =:g_{1}(n,k,t,s) .
	\end{align}
	
		\begin{ex}\label{buchonglizi} \cite{交叉t相交向量4}
		Let $n$, $k$,  $t$ and $s$ be positive integers with  $n\ge k+1\ge t+2$, and  $L\in  \qbinom{V}{t+1}$. The families 
		$$\mh_{1}(V,L;k)\\ ~ and ~ \\ \mathcal{M} (L;k,t) $$
		are $s$-almost cross-$t$-intersecting and the size of the intersection of all members in the union of $\mh_{1}(V,L;k)$ and $\mathcal{M}  _{1} (L;k,t)$ is less than $t$.
		
	\end{ex}
	
	The product of sizes of families in Example \ref{buchonglizi} is 
	$$q^{k-t}{t+1\brack 1}{n-t-1\brack k-t-1}{n-t-1\brack k-t} +{n-t-1\brack k-t-1}^{2} =:g_{2}(n,k,t).$$
	
	\begin{thm}\label{dingli5.2}
		Let $n$, $k$,  $t$ and $s$ be positive integers with $(k,t)\ne (4,2)$, $k\ge t+2$ and $n\geq 3k+3t+1+\log_{q} 13s$. Suppose that $\mathcal{F},\mathcal{G}\subseteq  \qbinom{V}{k}$ are $s$-almost cross-$t$-intersecting with $\dim\left (  {\textstyle \bigcap_{H\in \mathcal{F}\cup \mathcal{G}  }} H \right ) <t$. If $|\mathcal{F} ||\mathcal{G} |$ has the maximum value and $\left | \mathcal{F}  \right | \le \left | \mathcal{G}  \right | $, then one of the following hold.
		\begin{enumerate}[\normalfont(i)]
			\item If $k\le 2t $, then
			$|\mathcal{F} ||\mathcal{G} |\le g_{2}(n,k,t),$
			and the equality holds only if there exists $L\in  \qbinom{V}{t+1}$ such that $\mathcal{F} =\mh_{1}(V,L;k)$ and $\mathcal{G} ={\mathcal{M}} _{1} (L;k,t)$.
			\item If $k\ge 2t+1$, then
			$|\mathcal{F} ||\mathcal{G} |\le g_{1}(n,k,t,s),$
			and equality holds only if there exist $X\in  \qbinom{V}{k+1}$ and $E\in  \qbinom{X}{t}$ such that  $$\mf=\mh_{1}(V,E;k)\setminus \mathcal{A}~~and ~~ \mathcal{G}  =\mathcal{H} _{1} (V,E;k)\cup \mathcal{B}, $$
		\end{enumerate}
		where $\mathcal{A}\subseteq \mathcal{H} _{2} (X,E;k)$  and $\mb \subseteq {X\brack k}\setminus\mh_{1}(X,E;k)$ with $\left | \mathcal{A}  \right |= q^{(k-t)(k-t+1)} \qbinom{n-k-1}{k-t}-s $ and $\left | \mathcal{B}  \right |= \min\left\{s, q^{(k-t+1)} \qbinom{t}{1}\right\}$.
		
	\end{thm}

This paper proceeds with the following organization. Some inequalities needed in this paper are proved in Section \ref{s4}. Section \ref{section2} is dedicated to the proof of Theorem \ref{dingli1}. Theorem \ref{dingli5.2} are established in Section \ref{section3}.


		\section{Inequalities concerning Gaussian binomial coefficients}\label{s4}
	
	In this section, we begin with two lemmas which contain some inequalities about the Gaussian binomial coefficient and a formula for counting the number of some special subspaces. Then we present the proofs of the relevant inequality lemmas employed in this paper.

	\begin{lem}\label{yinli6.1}
		Let $n$ and $k$ be positive integers with $n>k$. Then the following hold.
		\begin{enumerate}[\normalfont(i)]
			\item ${k\b 1}\leq 2q^{k-1}$.
			\item $q^{n-k}<\frac{q^{n}-1}{q^{k}-1}<q^{n-k+1}$.
			\item  $q^{k(n-k)}<{n\brack k}<q^{k(n-k+1)}$.	
			\item  ${n\brack k}={n-1\brack k-1}+q^{k}{n-1\brack k} $.
			\item  Let $t$ be a positive integer. If  $k\geq t$ and $n\geq 2k-t+1$, then
			\begin{equation*}
				{k-t+1\brack 1}^{j-i}{n-j\brack k-j}\leq{n-i\brack k-i}
			\end{equation*}
			for any positive integers $i$ and $j$ with $i\leq j$.
		\end{enumerate}
	\end{lem}

		\begin{lem}\textnormal{(\cite[Lemma 2.1]{2403213})}\label{bu}
		Let $n$ and $k$ be non-negative integers with $n\ge k$. If $q\ge 2$, then
		$${n\brack k}\le \frac{7}{2} q^{k(n-k)} .$$		
			\end{lem}
Let $W$ be an $(e+l)$-dimensional vector space over $\mathbb{F} _{q} $, where $l,e\ge 1$, and let $L$ be a fixed $l$-subspace of $W$. We say that an $m$-subspace $U$ is type $(m,h)$ if $\dim(U\cap L)=h$. Define $\mathcal{M} (m,h;e+l,e)$ to be the set of all subspaces of $W$ with type $(m,h)$. Let $N'(m_{1},h_{1};m,h;e+l,e)$ to be the number of subspaces of $w$ with type $(m,h)$ containing a given subspace with type $(m_{1},h_{1})$.
	\begin{lem}\textnormal{(\cite{04162})}\label{jishu}
Let $0\le h_{1} \le h\le l$ and $0\le m_{1}-h_{1} \le m-h\le e$. Then $N'(m_{1},h_{1};m,h;e+l,e)\ne 0$. Moreover, if $N'(m_{1},h_{1};m,h;e+l,e)\ne 0$, then
$$N'(m_{1},h_{1};m,h;e+l,e)=q^{(l-h)(m-h-m_{1}+h_{1})} {e-(m_{1}-h_{1})\brack m-h-(m_{1}-h_{1})}{l-h_{1}\brack h-h_{1}}.$$ 
Observe that
$$|\mathcal{M} (m,h;e+l,e)|=N'(0,0;m,h;e+l,e)=q^{(l-h)(m-h)} {e\brack m-h}{l\brack h}.$$

	\end{lem}	
	
For positive integers $n,k,t,s$ and $x$, set 
\begin{align}\label{f1}
	f_{1}(n,k,t,s,x)={x\brack t} {k-t+1\brack 1} ^{x-t}{n-x\brack k-x} +s{x\brack t} \sum_{i=0}^{x-t-1} {k-t+1\brack 1}^{i}.
\end{align}	
	
	\begin{lem}\label{yinli6.2}
		Let $n$, $k$, $t$ and $l$ be positive integers with $k\ge t+1$ and $n\geq 2k+2t+1+\log_{q} 7sl$. Then the following hold.
		\begin{enumerate}[\normalfont(i)]
			\item The function $f_{1} (n,k,t,s,x)$ is decreasing as $x\in \left \{ t,t+1,...,k-1 \right \}$ increases.
			\item $f_{1} (n,k,t,s,k-1)>\frac{7l}{6}  s\qbinom{k}{t}\binom{2k-2t+2}{k-t+1} $.	
		\end{enumerate}
	\end{lem}
	\begin{proof} 
		(i) For each  $x\in \left \{ t,t+1,...,k-1 \right \}$, note that $$f(n,k,t,s,x)\geq {x\b t}{k-t+1\brack 1}^{x-t}{n-x\brack k-x}.$$ It follows from $n\geq 2k+2t+1+\log_{q} 7sl$ that
		
		\begin{equation*}
			\begin{aligned}
				\frac{f_{1}(n,k,t,s, x+1)}{f_{1}(n,k,t,s, x)}\leq  \frac{(q^{x+1}-1)(q^{k-x}-1){k-t+1\b1}}{(q^{x+1-t}-1)(q^{n-x}-1)}+\frac{s(q^{x+1}-1)\sum_{i=0}^{x-t}{k-t+1\b 1}^{i}}{(q^{x+1-t}-1){k-t+1\b 1}^{x-t}{n-x\b k-x}}.
			\end{aligned}
		\end{equation*}
		For $q\ge 2$, it is routine to check that 
		$$\frac{1}{2}q^{x+1-t}  \le q^{x+1-t}-1,~\sum_{i=0}^{x-t}{k-t+1\b 1}^{i}\le \frac{1}{2}{k-t+1\b 1} ^{x-t+1}. $$
		Then we have
		\begin{equation*}
			\begin{aligned}
				\frac{f_{1}(n,k,t,s, x+1)}{f_{1}(n,k,t,s, x)}\leq &\ \frac{q^{x+1}\cdot q^{k-t+1}  }{\frac{1}{2}\cdot q^{x+1-t}\cdot q^{n-k}   } +\frac{sq^{x+1} \cdot \frac{1}{2} q^{k-t} }{\frac{1}{2} q^{x+1-t}\cdot q^{(k-x)(n-x)} }\\
				\le &\ \frac{2}{q^{n-2k} } +\frac{s}{q^{n-2k}} <\frac{2}{7s}+ \frac{s}{7s}<1 .
			\end{aligned}
		\end{equation*}
		
		(ii) If $k=t+1$, then by $n\geq 2k+2t+1+\log_{q} 7sl$, we have $n-t\ge 2k+t+1+\log_{q} 7sl>\log_{q}(7slq^{t+1}-7sl+1 )$ and $\frac{6}{7l}\qbinom{n-t}{1}>6s\qbinom{t+1}{1}$.
		
		$$\frac{6}{7l}  f_{1} (n,k,t,s,k-1)=\frac{6}{7l}\qbinom{n-t}{1}> 6s\qbinom{t+1}{1}= s\qbinom{k}{t}\binom{2k-2t+2}{k-t+1}, $$ as desired. 
		
		Next assume $k\ge t+2$. Since $f_{1} (n,k,t,s,k-1) \ge \qbinom{k-1}{t}\qbinom{k-t+1}{1}^{k-t+1}\qbinom{n-k+1}{1} $, we have
		\begin{equation*}
			\begin{aligned}
				\frac{s\qbinom{k}{t}\binom{2k-2t+2}{k-t+1} }{ f_{1} (n,k,t,s,k-1)}\leq &\ \frac{s\qbinom{k}{t}\binom{2k-2t+2}{k-t+1}}{\qbinom{k-1}{t}\qbinom{k-t+1}{1}^{k-t-1}\qbinom{n-k+1}{1}} \le \frac{2s\cdot\frac{q^{k}-1 }{q^{k-t}-1 } \cdot(2+\frac{1}{k-t} )\cdot\binom{2k-2t}{k-t-1} }{(q^{k-t} )^{k-t-1}q^{n-k}  }\\		
				\leq &\  \frac{6s\cdot q^{t+1}q^{k-t-1}}{q^{n-k}}\left ( \frac{k-t+2}{q^{k-t+1}} \right )   ^{k-t-1} 
				\le  \frac{6s\cdot q^{k}}{q^{n-k}}\le\frac{6s}{7sl+2^{2t+1}}<\frac{6}{7l}.
			\end{aligned}
		\end{equation*}
		Hence, the desired result follows.
	\end{proof}
	
	\begin{lem}\label{yinli6.3}
		Let $n$, $k$, $t$ and $s$ be positive integers with $k\ge t+2$ and $n\geq 3k+3t+1+\log_{q} 13s$. Then the following hold.
		\begin{enumerate}[\normalfont(i)]
			\item $g_{1} (n,k,t,s)>\frac{25}{26} {k-t+1\brack 1}{n-t-1\brack k-t-1}{n-t\brack k-t}$.
			\item $g_{1} (n,k,t,s)>\left ( {k-t+1\brack 1}{n-t-1\brack k-t-1}+s \right ) ^{2}$.	
		\end{enumerate}
	\end{lem}
	
	\begin{proof} 
		(i) We follow the approach and arguments presented in \cite {2405263}. For each integer $i$ with $1\le i\le k-t+1$, we have
		\begin{align}\label{1}
			{n-t\brack k-t}-q^{(k-t)(k+1-t)}{n-k-1\brack k-t}=\sum_{i=1}^{k-t} (-1)^{i-1}a_{i}(k) ,
		\end{align}
		where $a_{i}(k)= q^{\binom{i}{2} }{k-t+1\brack i}{n-t-i\brack k-t-i}$ and $\frac{a_{i+1}(k) }{a_{i}(k)} <1 $. 
		Then it is clear that 
		\begin{equation*}\label{2}
			\begin{aligned}
				g_{1}(n,k,t,s)>&\  \left ( {n-t \brack k-t}-q^{(k-t)(k+1-t)}{n-k-1\brack k-t} \right ) {n-t\brack k-t}\\
				=&\ \sum_{i=1}^{k-t} (-1)^{i-1}q^{\binom{i}{2} }{k-t+1\brack i}{n-t-i\brack k-t-i}{n-t\brack k-t}\\
				>&\ \left ( {k-t+1\brack 1}{n-t-1\brack k-t-1}-q{k-t+1\brack 2}{n-t-2\brack k-t-2} \right ) {n-t\brack k-t}.
			\end{aligned}
		\end{equation*}
		
		From the above equation and $k\ge t+2$ and $n\geq 3k+3t+1+\log_{q} 13s$, we obtain
		
		\begin{equation*}
			\begin{aligned}
				\frac{	g_{1}(n,k,t,s)}{{k-t+1\brack 1}{n-t-1\brack k-t-1}{n-t\brack k-t}} >&\ 1-\frac{q{k-t+1\brack 2}{n-t-2\brack k-t-2}}{{n-t-1\brack k-t-1}{k-t+1\brack 1}} =1-\frac{(q^{k-t}-1 )(q^{k-t-1}-1)}{(q^{2}-1 )(q^{n-t-1}-1)} \\
				>&\ 1-\frac{1}{q^{n-2k+t+1} } >1-\frac{1}{13s\cdot2^{4t+2} } >1-\frac{1}{26}>\frac{25}{26}. 
			\end{aligned}
		\end{equation*}
		Then (i) holds.
		
		(ii)Note that
		\begin{equation*}
			\begin{aligned}
				\frac{\left ( {k-t+1\brack 1}{n-t-1\brack k-t-1}+s \right ) ^{2}}{{k-t+1\brack 1}{n-t-1\brack k-t-1}{n-t\brack k-t}}= \left (  1+\frac{s}{{k-t+1\brack 1}{n-t-1\brack k-t-1}} \right )  \cdot \frac{{k-t+1\brack 1}{n-t-1\brack k-t-1}+s}{{n-t\brack k-t}}.
		\end{aligned}
\end{equation*}
		By Lemma \ref{yinli6.1} (ii)(iii) and $k\ge t+2$, $n\geq 3k+3t+1+\log_{q} 13s$, we have the above expression
		\begin{equation*}
			\begin{aligned}
				\frac{\left ( {k-t+1\brack 1}{n-t-1\brack k-t-1}+s \right ) ^{2}}{{k-t+1\brack 1}{n-t-1\brack k-t-1}{n-t\brack k-t}} \le &\ \left (  1+\frac{s}{q^{k-t}\cdot q^{(k-t-1)(n-k)} }  \right ) \left (  \frac{q^{k-t+1} }{q^{n-k}}+\frac{s}{q^{(k-t)(n-k)}} \right ) \\
				\le &\  \left ( 1+\frac{s}{q^{2}\cdot q^{\log_{q}{13s} } }  \right ) \left (  \frac{1 }{q^{\log_{q}{13s}}}+\frac{s}{q^{2\log_{q}{13s}}} \right )\\
				\le &\ \left ( 1+\frac{1}{4\cdot13} \right )  \left ( \frac{1}{13}+  \frac{1}{13^{2} } \right ) =\frac{252}{2873} <1.
			\end{aligned}
		\end{equation*}
		This together with (i) yields (ii).
	\end{proof} 
	
		For positive integers $n,k,t,s$ and $x$, write
	
	\begin{align}\label{f2}
		f_{2}(n,k,t,s,x)={x-t\brack 1} {n-t-1\brack k-t-1}+q^{2(x-t)} {k+1-x\brack 1}^{2} {n-t-2\brack k-t-2} +2s  ,
	\end{align}
	\begin{align}\label{f3}
		f_{3}(n,k,t,s,x)={n-t\brack k-t}+q^{x-t+1}{k-t \brack 1}{t\brack 1}{n-x\brack k-x}+s .
	\end{align}
	
		\begin{lem}\label{yinli6.5}
		Let $n$, $k$, $t$ and $s$ be as in Lemma \ref{yinli6.3}. The function $f_{3} (n,k,t,s,x)$ is decreasing as $x\in \left \{  t+1, t+2,...,k-1  \right \}$ increases.
	\end{lem}
	
	\begin{proof} 
	By Lemma \ref{yinli6.1} (ii) and (iii), for each $x\in \left \{ t+1, t+2,...,k-1 \right \}$, we know
		\begin{equation*}
			\begin{aligned}
				f_{3}(n,k,t,s,x)-f_{3}(n,k,t,s,x+1)=&\ q^{x-t+1}{k-t \brack 1}{t\brack 1}{n-x\brack k-x}-q^{x-t+2}{k-t \brack 1}{t\brack 1}{n-x-1\brack k-x-1}\\
				\ge &\ q^{x-t+1}\left ( {n-x\brack k-x}-q{n-x-1\brack k-x-1} \right ) \\
				\ge &\ q^{x-t+1}\left ( \frac{q^{n-x}-1 }{q^{k-x}-1 } -q \right ) {n-x-1\brack k-x-1}\\
				\ge &\ q^{n-k}-q>0,
			\end{aligned}
		\end{equation*}
		as desired.
	\end{proof} 
	
	\begin{lem}\label{yinli6.4}
		Let $n$, $k$, $t$ and $s$ be as in Lemma \ref{yinli6.3}. Then the following hold.
		\begin{enumerate}[\normalfont(i)]
			\item The function $f_{3} (n,k,t,s,x)$ is increasing as $x\in \left \{  t+1, t+2,...,k-1  \right \}$ increases.
			\item $f_{2} (n,k,t,s,k)\left ( {n-t\brack k-t}+q^{2}{k-t \brack 1}{t\brack 1}{n-t-1\brack k-t-1}+s \right ) <g_{1}(n,k,t,s)$.
		\end{enumerate}
	\end{lem}
	
	\begin{proof} 
		(i)We have
		\begin{equation*}
			\begin{aligned}
				&\ f_{2}(n,k,t,s,x+1)-f_{2}(n,k,t,s,x)\\
				= &\ \left ( {x+1-t\brack 1}-{x-t\brack 1}  \right ) {n-t-1\brack k-t-1} +\frac{q^{2(x-t)}(q+1)-2q^{k-2t+1+x}  }{q-1}{n-t-2\brack k-t-2} \\
				\ge &\ q^{x-t} {n-t-1\brack k-t-1} +\frac{q^{x-2t} \left ( q^{x}(q+1)-2q^{k+1}  \right )   }{q-1}{n-t-2\brack k-t-2}.   
			\end{aligned}
		\end{equation*}
 For each $x\in \left \{ t+1, t+2,...,k-1 \right \}$ and $k\ge t+2$, we further obtain				
			\begin{equation*}
			\begin{aligned}			
				f_{2}(n,k,t,s,x+1)-f_{2}(n,k,t,s,x) \ge &\ q {n-t-1\brack k-t-1} +\frac{q^{k-1-2t}\left ( q^{t+1}(q+1)-2q^{k+1}  \right )  }{q-1}{n-t-2\brack k-t-2} \\
				\ge &\ {n-t-2\brack k-t-2}\left ( \frac{q^{n-t}-q }{q^{k-t-1} -1} +\frac{q^{3} +q^{2}-2q^{2(k-t)} }{q-1}  \right ) \\
				\ge &\ {n-t-2\brack k-t-2} \cdot \frac{q^{n-t}(q-1)+q-2q^{2}+3q^{4}-2q^{3k-3t-1}    }{\left ( q^{k-t-1} -1 \right )\left (  q-1 \right ) }.
            	\end{aligned}
            \end{equation*}
By $q\ge 2$, $q-2q^{2}+3q^{4}>0$ and $n\geq 3k+3t+1+\log_{q} 13s$, we have				
				\begin{equation*}
					\begin{aligned}			
				f_{2}(n,k,t,s,x+1)-f_{2}(n,k,t,s,x)\ge &\ \frac{{n-t-2\brack k-t-2}}{\left ( q^{k-t-1} -1 \right )\left (  q-1 \right ) } \cdot \left ( q^{n-t}-q^{3k-3t}  \right ) \\
				\ge &\ \frac{{n-t-2\brack k-t-2}q^{3k-3t}}{\left ( q^{k-t-1} -1 \right )\left (  q-1 \right ) } \cdot  (q^{n-3k+2t}-1 ) >0.
			\end{aligned}
		\end{equation*}
		
		(ii)Observe that
		\begin{equation*}
			\begin{aligned}
			&\ \frac{f_{2}(n,k,t,s,k)\left ( {n-t\brack k-t}+q^{2}{k-t \brack 1}{t\brack 1}{n-t-1\brack k-t-1}+s \right ) }{{k-t+1\brack 1}{n-t-1\brack k-t-1}{n-t\brack k-t}}\\
			 \le &\ \left ( \frac{{k-t\brack 1}}{{{k-t+1\brack 1}}}+\frac{q^{2(k-t)} {n-t-2\brack k-t-2}}{{{k-t+1\brack 1}{n-t-1\brack k-t-1}}} +\frac{2s}{{k-t+1\brack 1}{n-t-1\brack k-t-1}}  \right ) \left ( 1+\frac{q^{2}{k-t\brack 1}{t\brack 1}{n-t-1\brack k-t-1} }{{n-t\brack k-t}} +\frac{s}{{n-t\brack k-t}}  \right ).
			\end{aligned}
	\end{equation*}	 
	By Lemma \ref{yinli6.1} (ii)(iii), $k\ge t+2$ and $n\geq 3k+3t+1+\log_{q} 13s$, we have				
	\begin{equation*}
		\begin{aligned}	 		 
	&\ \frac{{k-t\brack 1}}{{{k-t+1\brack 1}}}+\frac{q^{2(k-t)} {n-t-2\brack k-t-2}}{{{k-t+1\brack 1}{n-t-1\brack k-t-1}}} +\frac{2s}{{k-t+1\brack 1}{n-t-1\brack k-t-1}} \\
	 \le &\  \frac{q^{k-t}-1 }{q^{k-t+1}-1}+  \frac{q^{2(k-t)} }{q^{k-t}}\cdot \frac{q^{k-t-1}-1}{q^{n-t-1}-1}+\frac{2s}{q^{k-t}\cdot q^{(k-t-1)(n-k)} }  \\
	 \le &\ \frac{1}{q}+\frac{1}{q^{n-2k+t} }   +\frac{2s}{q^{2}q^{n-k} }  \\
	 \le &\ \frac{1}{2}  +\frac{1}{13}+\frac{2}{2^{2}\cdot 13 }=\frac{8}{13}.
		\end{aligned}
\end{equation*}
And
	\begin{equation*}
	\begin{aligned}	 	
		1+\frac{q^{2}{k-t\brack 1}{t\brack 1}{n-t-1\brack k-t-1} }{{n-t\brack k-t}} +\frac{s}{{n-t\brack k-t}}
		\le &\ 1+\frac{q^{2}\cdot q^{k-t}\cdot q^{t} \cdot \left ( q^{k-t}-1 \right ) }{q^{n-k}-1 } +\frac{s}{q^{(k-t)(n-k)}}\\
		\le &\  1+\frac{1}{q^{n-2k-2} }+\frac{s}{q^{2(n-k)} } \\
		\le &\ 1  +\frac{1}{13}+\frac{1}{13^{2} }=\frac{171}{169}.
			\end{aligned}
	\end{equation*}
	Then we can obtain
	\begin{equation*}
		\begin{aligned}	 
				 \frac{f_{2}(n,k,t,s,k)\left ( {n-t\brack k-t}+q^{2}{k-t \brack 1}{t\brack 1}{n-t-1\brack k-t-1}+s \right ) }{{k-t+1\brack 1}{n-t-1\brack k-t-1}{n-t\brack k-t}}\le \frac{8}{13}\cdot \frac{171}{169}=\frac{1368}{2197} < \frac{25}{26}.
					\end{aligned}
			\end{equation*}

		It follows from Lemma \ref{yinli6.3} (i) that (ii) holds.
	\end{proof}

	\begin{lem}\label{6.5}
		Let $n$, $k$, $t$ and $s$ be as in Lemma \ref{yinli6.3}. Then the following hold.
		\begin{enumerate}[\normalfont(i)]
			\item $g_{1} (n,k,t,s)>\left ( {n-t\brack k-t}-q^{(k-t)(k+1-t)}{n-k-1\brack k-t}  \right ) \left ( {n-t\brack k-t}+ q^{(k-t+1)} \qbinom{t}{1} \right )$.
			\item If $k\le 2t$ and $t\ne 2$, then $g_{1}(n,k,t,s)<g_{2}(n,k,t)$.
		\end{enumerate}
	\end{lem}
	\begin{proof} 
		It is clear that 

		\begin{align}	\label{6.6}
			{n-t\brack k-t}-q^{(k-t)(k+1-t)}{n-k-1\brack k-t}=&\ \left | \left \{ F\subseteq {V\brack k} :W\subseteq F,\dim(F\cap M)\ge t+1 \right \}  \right | \nonumber\\
			\le &\ {k+1-t\brack 1}{n-t-1\brack k-t-1} .
		\end{align}

		(i) From (\ref{6.6}) and $n\geq 3k+3t+1+\log_{q} 13s$, we obtain  
		\begin{equation*}
			\begin{aligned}
				&\	g_{1}(n,k,t,s)- \left ( {n-t\brack k-t}-q^{(k-t)(k+1-t)}{n-k-1\brack k-t}  \right ) \left ( {n-t\brack k-t}+ q^{(k-t+1)} \qbinom{t}{1} \right) \\
				\ge &\ s{n-t\brack k-t}-q^{(k-t+1)} \qbinom{t}{1}\left ( {n-t\brack k-t}-q^{(k-t)(k+1-t)}{n-k-1\brack k-t} \right ) \\
				\ge &\ \left ( \frac{s(q^{n-t}-1 )}{q^{k-t}-1} -q^{k-t+1}{t\brack 1}{k-t+1\brack 1}  \right ) {n-t-1\brack k-t-1}  \\
				\ge &\ \left ( s q^{n-k}-q^{2k-t+2} \right ){n-t-1\brack k-t-1}> q^{2k+3t+1 } -q^{2k-t+2}>0.
			\end{aligned}
		\end{equation*}
		Then (i) holds.
		
		(ii) Divide proof  into two parts.
		
		\medskip
		\noindent{{\bf Case 1.} $k\le 2t-1$.}
		\medskip
		
		By Lemma \ref{yinli6.1} (ii)(iii), we have
		\begin{equation*}
			\begin{aligned}
				&\	\left ( \frac{q^{k-t}(q^{n-k}-1 ) }{q^{n-t}-1 } {t+1\brack 1}{n-t-1\brack k-t-1}-{t\brack 1}{n-t-1\brack k-t-1}-s \right ) {n-t\brack k-t}\\
				=&\ \left (  \left ( \frac{q^{k-t}(q^{n-k}-1 ) }{q^{n-t}-1 } {t+1\brack 1}-{t\brack 1}  \right ){n-t-1\brack k-t-1} -s \right )  {n-t\brack k-t}\\
				\ge &\ \left( \left (  \frac{1}{q}  {t+1\brack 1}-{t\brack 1}  \right ){n-t-1\brack k-t-1} -s \right ) {n-t\brack k-t}\\
				=&\ \left ( \frac{1}{q} {n-t-1\brack k-t-1} -s \right )  {n-t\brack k-t}>13s q^{2k+3t} -s>0
			\end{aligned}
		\end{equation*}
and		
		\begin{equation*}\label{6.7}
			\begin{aligned}
				{n-t-1\brack k-t-1}^{2}-s {t\brack 1}{n-t-1\brack k-t-1}-s^{2}=&\ {n-t-1\brack k-t-1}\left ( {n-t-1\brack k-t-1}-s {t\brack 1} \right ) -s^{2}\\
				> &\ q^{n-k} \left ( q^{n-k}-s\cdot q^{t} \right ) -s^{2} \\
				>&\ 13s(13sq^{t} -s q^{t})-s^{2} =156s^{2} q^{t}-s^{2}>0.
			\end{aligned}
		\end{equation*}
		Note that $k\le 2k-1$, it follows from Lemma \ref{6.5} that 
		\begin{equation*}
			\begin{aligned}
				g_{1}(n,k,t,s)\le &\ \left ( {k-t+1\brack 1}{n-t-1\brack k-t-1}+s \right )\left (  {n-t\brack k-t}+s \right ) \\
				\le &\ {t\brack 1}{n-t-1\brack k-t-1}{n-t\brack k-t}+s{t\brack 1}{n-t-1\brack k-t-1}+s{n-t\brack k-t}+s^{2}:=n_{1}(n,k,t,s).
			\end{aligned}
		\end{equation*}
		We have 
		\begin{equation*}
			\begin{aligned}
				&\	g_{2}(n,k,t)-g_{1}(n,k,t,s)\ge g_{2}(n,k,t)-n_{1}(n,k,t,s) \\
			= &\ q^{k-t}{t+1\brack 1}{n-t-1\brack k-t-1}{n-t-1\brack k-t} +{n-t-1\brack k-t-1}^{2} -{t\brack 1}{n-t-1\brack k-t-1}{n-t\brack k-t}\\
				&\ -s{t\brack 1}{n-t-1\brack k-t-1}-s{n-t\brack k-t}-s^{2}\\
				=&\ \left ( \frac{q^{k-t}(q^{n-k}-1 ) }{q^{n-t}-1 } {t+1\brack 1}{n-t-1\brack k-t-1}-{t\brack 1}{n-t-1\brack k-t-1}-s \right ) {n-t\brack k-t}+{n-t-1\brack k-t-1}^{2}\\
				&\ -s {t\brack 1}{n-t-1\brack k-t-1}-s^{2} >0.
			\end{aligned}
		\end{equation*}
		as desired.
		
		\medskip
		\noindent{{\bf Case 2.} $k=2t$ and $t\ge 3$.}
		\medskip
		
		For convenience, write
		\begin{equation*}
			\begin{aligned}
				&\	m_{1}(n,k,t)= {n-t\brack k-t}-q^{(k-t)(k+1-t)}{n-k-1\brack k-t}+s, \\
				&\ m_{2}(n,k,t)= {n-t\brack k-t}+ q^{(k-t+1)} \qbinom{t}{1},\\
				&\ g_{3}(n,k,t,s)= {k-t+1\brack 1}{n-t+1\brack k-t+1}-q^{(k-t-1)(k-t-2)+1} {n-t-1\brack k-t-2}{k+1-t\brack 2}+s .
			\end{aligned}
		\end{equation*}
		
		By \cite[Lemma 5.1]{交叉t相交向量4}, we have $m_{1}(n,k,t) \le g_{3}(n,k,t,s)$. And $g_{1} (n,k,t,s)\le m_{1}(n,k,t)m_{2}(n,k,t)\le g_{3}(n,k,t,s)m_{2}(n,k,t)$ . Therefore, we only need to prove $g_{3}(n,k,t,s)m_{2}(n,k,t)<g_{2}(n,k,t)$.
		
		When $t\ge 3$, by Lemma \ref{yinli6.1} (ii) (iii), Lemma \ref{bu} and $n\geq 2k+2t+1+\log_{q} 7s$, we have
		\begin{equation*}
			\begin{aligned}
				&\ {n-t-1\brack t-1}^{2}- \left ( {t+1\brack 1}{n-t-1\brack t-1}+s \right )q^{t+1}{t\brack 1} -s{n-t\brack t} \\
				\ge &\ \left ( q^{2(t-1)(n-2t)}-2q^{t} \cdot \frac{7}{2} q^{(t-1)(n-2t)} +s \right ) q^{t+1}\cdot 2q^{t-1}-\frac{7}{2}s \cdot q^{t(n-2t)}  \\
				= &\ q^{2(t-1)(n-2t)}-14q^{3t+(n-2t)(t-1)} -2sq^{2t}- \frac{7}{2}s \cdot q^{t(n-2t)}  \\
				\ge 	&\ q^{t(n-2t)}\left ( q^{n-2t}-14q^{5t-n} -\frac{7}{2}s  \right )-2sq^{2t}   \\
				> &\ q^{t\left (7t+1+\log_{q}{13s} \right ) }  \left ( q^{7t+1+\log_{q}{13s} }-\frac{14}{q^{\log_{q}{13s}}} -\frac{7}{2}s  \right )-2sq^{2t} \\
				> &\ (13s)^{t} \cdot q^{7t^{2}+t } -2sq^{2t} > 0.
			\end{aligned}
		\end{equation*}

		By \cite[Lemma 5.3]{交叉t相交向量4}, we have
		\begin{equation*}
			\begin{aligned}
		\frac{q^{t} (q^{n-2t}-1 )}{q^{n-t}-1} -1+q^{(t-1)(t-2)+1}\cdot  \frac{q^{t}-1 }{q^{2}-1}\frac{{n-2t-1\brack t-2}}{{n-t-1\brack t-1}}>0
			\end{aligned}
	\end{equation*}
	Then we can obtain
		\begin{equation*}
			\begin{aligned}
				&\ \frac{g_{2}(n,k,t)-g_{3}(n,k,t,s)m_{2}(n,k,t)}{{t+1\brack 1}{n-t-1\brack t-1}{n-t\brack t}}  \\
				= &\ \frac{q^{t} {n-t-1\brack t}}{{n-t\brack t}}  -\left ( 1-\frac{q^{(t-1)(t-2)+1}{n-2t-1\brack t-2} {t+1\brack 2}}{{t+1\brack 1}{n-t-1\brack t-1}} +\frac{s}{{t+1\brack 1}{n-t-1\brack t-1}}  \right )\left ( 1+\frac{q^{t+1}{t\brack 1} }{{n-t\brack t}}  \right )+\frac{{n-t-1\brack t-1}^{2} }{{t+1\brack 1} {n-t-1\brack t-1} {n-t\brack t}}   \\
				\ge &\ \frac{q^{t} (q^{n-2t}-1 )}{q^{n-t}-1} -1+q^{(t-1)(t-2)+1}\cdot  \frac{q^{t}-1 }{q^{2}-1}\frac{{n-2t-1\brack t-2}}{{n-t-1\brack t-1}}-\frac{s}{{t+1\brack 1}{n-t-1\brack t-1} }+\frac{{n-t-1\brack t-1}^{2} }{{t+1\brack 1} {n-t-1\brack t-1} {n-t\brack t}}\\
				&\ -    \left ( 1+\frac{s}{{t+1\brack 1}{n-t-1\brack t-1} }\right ) \frac{q^{t+1}{t\brack 1}}{{n-t\brack t}}  \\
				\ge &\ \frac{{n-t-1\brack t-1}^{2} }{{t+1\brack 1} {n-t-1\brack t-1} {n-t\brack t}} -\left ( 1+\frac{s}{{t+1\brack 1}{n-t-1\brack t-1}}  \right ) \frac{q^{t+1}{t\brack 1}}{{n-t\brack t}}  -\frac{s}{{t+1\brack 1}{n-t-1\brack t-1}}\\
				=&\ \frac{{n-t-1\brack t-1}^{2}- \left ( {t+1\brack 1}{n-t-1\brack t-1}+s \right )q^{t+1}{t\brack 1} -s{n-t\brack t}}{{t+1\brack 1}{n-t-1\brack t-1}{n-t\brack t}} \ge 0.
			\end{aligned}
		\end{equation*}
		Then (ii) holds.
	\end{proof}

	\section{Proof of Theorem \ref{dingli1}}\label{section2}

	Before establishing Theorem \ref{dingli1}, we present a series of  properties associated with s-almost cross-$t$-intersecting families.

	Let $\mf\subseteq{V\b k}$. A subspace $T$  of $V$ is called a \textit{$t$-cover} of $\mf$ if $\dim(T\cap F)\geq t$ for any $F\in \mf$. The family of all $t$-covers with size $\tau_{t}(\mf)$ is denoted by $\mathcal{T}(\mf) $.
	Define the \textit{$t$-covering number} $\tau_{t}(\mf)$ of $\mf$ as the minimum dimension of a $t$-cover of $\mf$. For a subspace $G$ of $V$, write
	$$\mf_{G}=\left\{F\in\mf:G\subseteq F\right\}.$$


	From now on, we assume that $s$-almost cross-$t$-intersecting families are non-empty.
	\begin{lem}\label{yinli2.2}
		Let $n$, $k$, $t$ and $s$ be positive integers satisfying $k\geq t+1$ and $n\geq 2k$. Let $\mathcal{F}, \mathcal{G}\subseteq {V\brack k}$ are $s$-almost cross-$t$-intersecting families with $\tau_{t}(\mathcal{G})\le k$. If $H$ is a non-empty subspace of $V$ with $\dim(H)\le  \tau_{t}(\mathcal{G})$, then 
		$$\left| \mf_{H}\right|\leq {k-t+1\b 1}^{\tau_{t}(\mathcal{G})-\dim(H)}\qbinom{n-\tau_{t}(\mathcal{G})}{k-\tau_{t}(\mathcal{G})}+s\sum_{i=0}^{\tau_{t}(\mg)-\dim(H)-1}{k-t+1\b 1}^{i}.$$
	\end{lem}
	\begin{proof}

				Let $h$ be a non-negative integer. Suppose that $\mathcal{A}\subseteq {V\brack k}$ and $B\in {V\brack k} $ satisfy $|\mathcal{D}_{\mathcal{A} } (B;t) | \le s$. If $H\subseteq {V\brack h}$ with $\dim (H\cap B)=m$<t, then exists a subspace $R$ of $V$ with $H\subsetneq  R$ such that 
			\begin{equation}\label{DY}
			\left| \mathcal{F}_{H}\right|\leq {k-t+1\brack 1}^{t-m}\left| \mathcal{F}_{R} \right|+s.
			\end{equation}

		If $\mathcal{A} =\mathcal{D}_{\mathcal{A} } (B;t)$, then the required results holds. Next assume $\mathcal{A} \setminus  \mathcal{D}_{\mathcal{A} } (B;t)\ne \emptyset $.
		
		Let $\mi=\ma\bs\md_{\ma}(A;t)$. Then $\dim\left ( B\cap I \right ) \ge t$ for any $I\in \mi$. By \cite[Lemma 2.4]{交叉t相交向量4}, there exists an $(h+t-m)$-subspace $R$ with $H\subseteq R$ such that 
		\begin{equation}\label{shi1}
			\left| \mathcal{I}_{H}\right|\leq {k-m\brack t-m}\left| \mathcal{I}_{R} \right|\leq {k-m\brack t-m}\left| \mathcal{F}_{R} \right|.
		\end{equation}
		
		The fact $\mf_{A}\bs \mi_{H}\subseteq \mf\bs \mi= \mathcal{D}_{\mathcal{A} } (B;t)$ implies
		\begin{equation}\label{shi2}
			\left|\ma_{H}\right|\leq \left|\mi_{H}\right|+s.
		\end{equation}

		It is routine to verify that
		$$\frac{q^{k-m}-1}{q^{t-m}-1}\leq \frac{q^{k-m-1}-1}{q^{t-m-1}-1}\leq \cdots\leq \frac{q^{k-t+1}-1}{q-1}.$$
		Hence  ${k-m\b t-m}\leq {k-t+1\b 1}^{t-m}$.
		Then from (\ref{shi1}) and (\ref{shi2}), the (\ref{DY}) holds.

		We now prove the Lemma \ref{yinli2.2}. If $\dim(H)=\tau_{t}(\mathcal{G})$, then the desired result follows from $\left | \mathcal{F}_{H}   \right | \le \qbinom{n-\dim(H)}{k-\dim(H)}=\qbinom{n-\tau_{t}(\mathcal{G}}{k-\tau_{t}(\mathcal{G}}$. Next assume $\dim(H)<\tau_{t}(\mathcal{G})$.
		
		From $\dim(H)<\tau_{t}(\mathcal{G})$, we know $\dim(H\cap G_{1})< t$ for some $ G_{1}\in \mg$. Since $\mathcal{F}$ and $\mathcal{G}$ are $s$-almost cross-$t$-intersecting families, we have $\left | \mathcal{D}_{F} (G;t)  \right | \le s$. By (\ref{shi1}), there exists a subspace $H_{2}$ of $V$ with $H\subsetneq H_{2}$ such that 
		$$\left|\mf_{H}\right|\leq {k-t+1\b 1}^{\dim(H_{2})-\dim(H)}\left|\mf_{H_{2}}\right|+s.$$
		 Using (\ref{DY}) repeatedly,
		we finally get an ascending chain of subspaces $H=:H_{1}\subsetneq H_{2}\subsetneq\ldots \subsetneq H_{u}$  such that 
		$\dim(H_{u-1})<\tau_{t}(\mathcal{F})\leq \dim(H_{u})$ and 
		$$\left|\mf_{H_{i}}\right|\leq {k-t+1\brack 1}^{\dim(H_{i+1})-\dim(H_{i})}\left| \mf_{H_{i+1}}\right|+s, \quad 1\leq i\leq u-1.$$
		This implies 
		\begin{equation*}
			\left| \mf_{H}\right|\leq {k-t+1\brack 1}^{\dim(H_{u})-\dim(H)}\left| \mf_{H_{u}}\right|+s\sum_{i=1}^{u-1}{k-t+1\brack 1}^{\dim(H_{i})-\dim(H)}.
		\end{equation*}
		
	Note that
			\begin{equation*}
		s\sum_{i=1}^{u-1}{k-t+1\brack 1}^{\dim(H_{i})-\dim(H)}\le   s\sum_{i=1}^{\tau_{t}(\mathcal{G} )-\dim(H)  -1}{k-t+1\brack 1}^{i}.
	\end{equation*}
		
		By lemma \ref{yinli6.1}(v), we have
			\begin{equation*}
					\begin{aligned}
		\left| \mf_{H}\right|\leq &\ {k-t+1\b 1}^{\dim(H_{u})-\dim(H)}\qbinom{n-\dim(H_{u})}{k-\dim(H_{u})}+s\sum_{i=0}^{\tau_{t}(\mg)-\dim(H)-1}{k-t+1\b 1}^{i}\\
		\leq &\ {k-t+1\b1}^{\tau_{t}(\mathcal{G})-\dim(H)}\qbinom{n-\tau_{t}(\mathcal{G})}{k-\tau_{t}(\mathcal{G})}+s\sum_{i=0}^{\tau_{t}(\mg)-\dim(H)-1}{k-t+1\b 1}^{i}.
			\end{aligned}
	\end{equation*}
		The desired result follows.
	\end{proof}
	
In the following, we establish some upper bounds for the product of sizes of $s$-almost cross-$t$-intersecting families in terms of their $t$-covering numbers.

	 \begin{lem}\label{yinli2.3}
		Let $n$, $k$, $t$ and $s$ be positive integers satisfying $k\geq t+1$ and $n\geq 2k$. Let $\mathcal{F}, \mathcal{G}\subseteq {V\brack k}$ are $s$-almost cross-$t$-intersecting families. If $\tau _{t}(\mathcal{F} )\le k$ and $\tau _{t}(\mathcal{G} )\le k$, then
		$$\left | \mathcal{F}  \right |  \left | \mathcal{G}  \right |   \le f_{1}(n,k,t,s,\tau _{t}(\mathcal{F} ) ) f_{1}(n,k,t,s,\tau _{t}(\mathcal{G} ) ) .$$
	\end{lem}	
	\begin{proof}
		Let $T_{f}$ and $T_{g}$ are minimal $t$-covers of $\mathcal{F}$ and $\mathcal{G}$, respectively. Then $\mathcal{F}= {\textstyle \bigcup_{H\in {T_{f} \brack t}}}  \mathcal{F} _{H} $ and $\mathcal{G}= {\textstyle \bigcup_{H\in {T_{g} \brack t}}}  \mathcal{G} _{H} $. It follows from Lemma \ref{yinli2.2} that
		$$\left | \mathcal{F}  \right | \le {\tau _{t}  (\mathcal{F} )\brack t} \left ( {k-t+1\brack 1} ^{\tau _{t}  (\mathcal{G} )-t} {n-\tau _{t}  (\mathcal{G} )\brack k-\tau _{t}  (\mathcal{G} )}+s \sum_{i=0}^{\tau _{t}  (\mathcal{G} )-t-1} {k-t+1\brack 1}^{i} \right ) $$
		$$\left | \mathcal{G}  \right | \le {\tau _{t}  (\mathcal{G} )\brack t} \left ( {k-t+1\brack 1} ^{\tau _{t}  (\mathcal{F} )-t} {n-\tau _{t}  (\mathcal{F} )\brack k-\tau _{t}  (\mathcal{F} )}+s \sum_{i=0}^{\tau _{t}  (\mathcal{F} )-t-1} {k-t+1\brack 1}^{i} \right ) .$$
The desired result holds.			
	\end{proof}

	 \begin{lem}\label{yinli2.4}
		Let $n$, $k$, $t$ and $s$ be positive integers satisfying $k\geq t+1$ and $n\geq 2k$. If $\mathcal{F},\mathcal{G}\subseteq {V\brack k}$ are $s$-almost cross-$t$-intersecting families with $\tau_t(\mg)\geq k+1$, then 
		$$|\mf||\mg|\leq s{k\brack t}{n-t\brack k-t}\binom{2k-2t+2}{k-t+1} +s^{2} \binom{2k-2t+2}{k-t+1}.$$
	\end{lem}	
	\begin{proof}
		Set $V_{1}=\mf$. Choose $F_{i}$, $G_{i}$ and $V_{i+1}$ by repeating the following steps:
		$$F_{i}\in V_{i},\ G_{i}\in\md_{\mg}(F_{i};t),\ V_{i+1}=V_{i}\bs\md_{\mf}(G_{i};t).$$
		Since $\left| V_{i+1}\right|< \left| V_{i}\right|$, there exists a positive integer $m$ such that $V_{m+1}=\emptyset$.	Finally, we  get two sequences of $k$-subspaces $F_{1},F_{2},\ldots,F_{m}$ and $G_{1}, G_{2},\ldots,G_{m}$ such that
		\begin{enumerate}[\normalfont(i)]
			\item $\dim(F_{i}\cap G_{i})<t$ for any $1\le i\le m$;
			\item $\dim(F_{i}\cap G_{j})\geq t$ for any $1\le j<i\le m$;
			\item $\mf=\cup _{i=1}^{m}\md_{\mf}(G_{i};t)$.
		\end{enumerate}

		By \cite[Theorem 5]{2406291}, two sequences of $k$-subspaces $F_{1},F_{2},\ldots,F_{m}$ and $G_{1}, G_{2},\ldots,G_{m}$ satisfying  (i) and (ii) must have $m\leq \binom{2k-2t+2}{k-t+1}$.	Notice that $\mf$ is $s$-almost $t$-intersecting. 
		From (iii), we obtain 
		$$\left|\mf\right|\leq \sum_{i=1}^{m}\left|\md_{\mf}(B_{i};t)\right|\leq s \binom{2k-2t+2}{k-t+1}.$$
		Let $F\in \mf$, then $F$ is $t$-cover of $ \mathcal{G} \setminus \mathcal{D_{G} } (F;t)$. We further derive $ \mathcal{G} \setminus \mathcal{D_{G} } (F;t)\subseteq  {\textstyle \bigcup_{H\in {F\brack t}}}\mathcal{G }_{H} $, which implies
		$$\left | \mathcal{G}  \right |  =| \mathcal{G} \setminus \mathcal{D_{G} } (F;t)|+|\mathcal{D_{G} } (F;t)|\le {k\brack t}{n-t\brack k-t}+s.$$
		This together with the upper bound of $|\mf|$ yields the desired results.
	\end{proof}

		 \begin{lem}\label{yinli2.5}
		Let $n$, $k$, $t$ and $s$ be positive integers satisfying $k\geq t+1$ and $n\geq 2k+2t+1+\log_{q} 7s$. Suppose that $\mathcal{F},\mathcal{G}\subseteq {V\brack k}$ are $s$-almost cross-$t$-intersecting families. If $\left ( \tau _{t} (\mathcal{F} ),\tau _{t} (\mathcal{G} ) \right ) \ne (t,t)$, then $|\mf||\mg|<{n-t\brack k-t}^{2}$.
	\end{lem}	
	\begin{proof}
		Based on the $t$-covering numbers of $\mf$ and $\mg$, we divide our proof into two cases.
		
		\medskip
		\noindent \textbf{Case 1.} $\tau_{t}(\mathcal{F})\leq k$ and $\tau_{t}(\mathcal{G})\leq k$. 
		\medskip
		
		By Lemma \ref{yinli2.3}, we have
		$$\left | \mathcal{F}  \right |  \left | \mathcal{G}  \right |   \le f_{1}(n,k,t,s,\tau _{t}(\mathcal{F} )  f_{1}(n,k,t,s,\tau _{t}(\mathcal{G} ) ) .$$
		This together with $\left ( \tau _{t} (\mathcal{F} ),\tau _{t} (\mathcal{G} ) \right ) \ne (t,t)$ and Lemma \ref{yinli6.2} (i) yields
		$$\left | \mathcal{F}  \right |  \left | \mathcal{G}  \right |   \le f_{1}(n,k,t,s,t ) f_{1}(n,k,t,s,t+1)<f_{1}(n,k,t,s,t )^{2}={n-t\brack k-t}^{2},$$
		as desired.

		\medskip
\noindent \textbf{Case 2.} $\tau_{t}(\mathcal{F})\ge k+1$ or $\tau_{t}(\mathcal{G})\ge k+1$. 
\medskip

W.l.o.g., we may assume that $\tau_{t}(\mathcal{G})\ge k+1$. From Lemma \ref{yinli2.4}, we obtain
$$|\mf||\mg|\leq s{k\brack t}{n-t\brack k-t}\binom{2k-2t+2}{k-t+1} +s^{2} \binom{2k-2t+2}{k-t+1}.$$ It follows from Lemma \ref{yinli6.2} (i) and (ii) that
$$\frac{6}{7}f_{1}(n,k,t,s,t){n-t\brack k-t}+\frac{6s}{7}f_{1}(n,k,t,s,t) =\frac{6}{7}{n-t\brack k-t}^{2}+\frac{6s}{7}{n-t\brack k-t}.$$
This combining with $6s<7s=q^{\log_{q}7s}  <q^{(k-t)(n-k)} <{n-t\brack k-t}$ yields $|\mf||\mg|<{n-t\brack k-t}^{2}$.
	\end{proof}
	
	To characterize the maximal $s$-almost cross-$t$-intersecting families, we proceed by proving a property of their minimum t-covers.

	\begin{lem}\label{yinli2.6}
		Let $n$, $k$ and $t$ be positive integers with $k\geq t+1$ and $n\geq 2k+s-1$. Suppose that $\mathcal{F},\mathcal{G}\subseteq {V\brack k}$ are maximal $s$-almost cross-$t$-intersecting families. If $ \tau _{t} (\mathcal{F} )\le k, \tau _{t} (\mathcal{G} ) \le k$, then $\mathcal{T}  (\mathcal{F} )$ and $\mathcal{T}  (\mathcal{G} )$ are cross-$t$-intersecting.
	\end{lem}	
	\begin{proof}
		It is sufficient to show $\dim(T_{f}\cap T_{g})\ge t$ for any $T_{f}\in \mathcal{T}  (\mathcal{F} )$, $T_{g}\in \mathcal{T}  (\mathcal{G} )$. 
		
		If $\tau _{t} (\mathcal{F} )=k$ or $\tau _{t} (\mathcal{G} )=k$, we may without loss of generality assume the first holds. Then $T_{f}\in \mathcal{T}(\mathcal{G} ) \subseteq \mathcal{G} $ by the maximality of $\mathcal{G}$. Since $T_{g}$ is a $t$-cover of $\mathcal{G}$, we have $\dim(T_{f}\cap T_{g})\ge t$, as desired. In the following, assume that $\mathcal{T} (\mathcal{F} )\le k$ and $\mathcal{T} (\mathcal{G} )\le k$.
		
		Suppose for contradiction that $\dim(T_{f}\cap T_{g})\le t-1$. From Lemma \ref{jishu}, there exists $U\subseteq {V\brack k}$ such that $U\cap(T_{f}+T_{g})=T_{g}$. Set 
		$$\mathcal{W} =\left \{ W\in {V\brack k}:W\cap (U+T_{f})=T_{f} \right \}. $$
	 Based on
		$n-\dim(U+ T_{f})\ge k-\tau _{t}(\mathcal{F} )> 0$. Then by Lemma \ref{jishu}, we get
		\begin{equation*}
			\begin{aligned}
				 \left | \mathcal{W}   \right |= q^{\left ( \dim(U+T_{f} )-\tau _{t}(\mathcal{F} ) \right ) (k-\tau _{t}(\mathcal{F} ))}{n-\dim(U+T_{f} )\brack k-\tau _{t}(\mathcal{F} )} .
			\end{aligned}
		\end{equation*}
It follows from Lemma \ref{yinli6.1} that 
$$\left | \mathcal{W}  \right | \ge q^{(k-\tau _{t}(\mathcal{F} ) )^{2} } \cdot q^{(k-\tau _{t}(\mathcal{F} )(n-\dim(U+T_{f} ) -k+\tau _{t}(\mathcal{F})}) \ge q\cdot q^{s-1}>s .$$		
	
	For each $W\in \mathcal{W}$, we have 
	$$W\cap U=W\cap U\cap (U+T_{f})=T_{f}\cap U=(T_{f}+T_{g})\cap T_{f}\cap U=T_{f}\cap T_{g}.$$
	Which implies $\dim(W\cap U)\le t-1$. Since $\mathcal{F}, \mathcal{G}$ are $s$-almost cross-$t$-intersecting families, we know $U\notin \mathcal{F}$ or $\mathcal{W}\nsubseteq  \mathcal{G}$. Since $\mathcal{F}$, $\mathcal{G}$ are maximal and $T_{f}\subseteq W$, $T_{g}\subseteq U$, we know $U\in \mathcal{F}$, $\mathcal{W} \subseteq  \mathcal{G}$, a contradiction. Then the desired results hold.
		\end{proof}
	

	\begin{proof}[\bf Proof of Theorem \ref{dingli1}] Observe that $\mathcal{H}_{1}(V,E;k)$ and $\mathcal{H}_{1}(V,E;k)$ are $s$-almost cross-$t$-intersecting. Then $|\mathcal{F} | |\mathcal{G} |\ge {n-t\brack k-t} ^{2}  $. 
		
	It follows from Lemma \ref{yinli2.5} that $\left ( \tau _{t} (\mathcal{F} ),\tau _{t} (\mathcal{G} ) \right ) = (t,t)$. Pick $E_{f}\in  \tau _{t} (\mathcal{F} )$ and $E_{g}\in  \tau _{t} (\mathcal{G} )$. Then $\mathcal{F}\subseteq \mathcal{H}_{1}(V,E_{f};k)$ and $\mathcal{G}\subseteq \mathcal{H}_{1}(V,E_{g};k)$. Since $|\mf||\mg|$ takes the maximum value, we know that $\mf$ and $\mg$ are maximal. It follows from Lemma \ref{yinli2.6} that $E_{f}=E_{g}=:E$. Using the maximality of $\mf$ and $\mg$ again, we conclude $\mf=\mg=\mathcal{H}_{1}(V,E;k)$.
	\end{proof}
		

	\section{Proof of Theorem \ref{dingli5.2}}\label{section3}
	In this section, we first present some upper bounds for the product of cardinalities.

	\begin{lem}\label{yinli3.2}
		Let $n$, $k$, $t$ and $s$ be positive integers with $k\geq t+2$ and $n\geq  2k+3t+1+\log_{q}{13s}$. Suppose that $\mathcal{F}, \mathcal{G}\subseteq {V\brack k}$ are $s$-almost cross-$t$-intersecting families. If $(\tau_{t}(\mf),\tau_{t}(\mg))\notin \left \{ (t,t),(t,t+1),(t+1,t) \right \}$, then $|\mf||\mg| <g_{1}(n,k,t,s) $.
	\end{lem}
	\begin{proof}
		By Lemma \ref{yinli6.3} (i), it is sufficient to show
		$$\frac{|\mf||\mg|}{{k-t+1\brack 1}{n-t-1\brack k-t-1}{n-t\brack k-t}} <\frac{25}{26} .$$
			\medskip
		\noindent \textbf{Case 1.} $\tau_{t}(\mathcal{F})\leq k$ and $\tau_{t}(\mathcal{G})\leq k$. 
		\medskip
	
		If $\tau_{t}(\mathcal{F})\geq t+1$ and $\tau_{t}(\mathcal{G})\geq t+1$, then by Lemma \ref{yinli2.3} and Lemma \ref{yinli6.2}(i), we have
		$$|\mf||\mg|\le f_{1}(n,k,t,s,t+1) ^{2} =\left ( {t+1\brack 1}{k-t+1\brack 1}{n-t-1\brack k-t-1}+s{t+1\brack 1} \right ) ^{2}.$$
		
	By Lemma \ref{yinli6.1} (ii)(iii) and $n\geq 2k+2t+1+\log_{q} 7s$, we have
				\begin{equation*}
			\begin{aligned}
				\frac{|\mf||\mg|}{{k-t+1\brack 1}{n-t-1\brack k-t-1}{n-t\brack k-t}} \le &\ {t+1\brack 1}^{2} \left ( \frac{{k-t+1\brack 1}{n-t-1\brack k-t-1}}{{n-t\brack k-t}} +\frac{2s}{{n-t\brack k-t}}+\frac{s^{2}}{{k-t+1\brack 1}{n-t-1\brack k-t-1}{n-t\brack k-t}} \right ) \\
			\le &\ q^{2t+2}\left ( \frac{q^{k-t+1} }{n-k}+\frac{2s}{q^{(k-t)(n-k)} } +\frac{s^{2} }{q^{k-t}\cdot q^{(k-t-1)(n-k)} \cdot q^{(k-t)(n-k)}}      \right )  \\
			\le	&\ \frac{1}{q^{\log_{q}{13s}+2t-2}  } +\frac{2s}{q^{2\log_{q}{13s} } }+\frac{s^{2} }{q^{2 }q^{\log_{q}{13s} }q^{2\log_{q}{13s} }} \\
			 \le &\ \frac{1}{13s}+\frac{2s}{(13s)^{2}} +\frac{s^{2}}{4(13s)(13s)^{2}} =\frac{729}{8788}<\frac{25}{26}.
			\end{aligned}
		\end{equation*}
		
			If $\tau_{t}(\mathcal{F})= t$ or $\tau_{t}(\mathcal{G})= t$, then by $(\tau_{t}(\mf),\tau_{t}(\mg))\notin \left \{ (t,t),(t,t+1),(t+1,t) \right \}$, we know $\tau_{t}(\mathcal{F})\ge t+2$ or $\tau_{t}(\mathcal{G})\geq t+2$. It follows from Lemma \ref{yinli2.3} and Lemma \ref{yinli6.2}(i) that
					\begin{equation*}
		\begin{aligned}
			|\mf||\mg|\le &\ f_{1}(n,k,t,s,t) f_{1}(n,k,t,s,t+2)\\
			=&\ {n-t\brack k-t}\left ( {{t+2\brack 2}{k-t+1\brack 1}^{2}{n-t-2\brack k-t-2}+s{t+2\brack t}}\left ( 1+{k-t+1\brack 1} \right ) \right )\\
			=&\ {n-t\brack k-t}{t+2\brack 2} \left ( {k-t+1\brack 1}^{2}{n-t-2\brack k-t-2}+s+s{k-t+1\brack 1} \right ) .
		\end{aligned}
	\end{equation*}		
		
Then we have
		\begin{equation*}
			\begin{aligned}
				\frac{|\mf||\mg|}{{k-t+1\brack 1}{n-t-1\brack k-t-1}{n-t\brack k-t}} \le &\ {t+2\brack 2} \left ( \frac{{k-t+1\brack 1}{n-t-2\brack k-t-2}}{{n-t-1\brack k-t-1}} +\frac{s}{{k-t+1\brack 1}{n-t-1\brack k-t-1}}+\frac{s}{{n-t-1\brack k-t-1}} \right ) \\
			<	&\ \frac{q^{2(t+1)} \cdot q^{k-t+1}}{q^{n-k}} +\frac{s\cdot q^{2(t+1)}}{q^{k-t}\cdot q ^{(k-t-1)(n-k)}}+\frac{sq^{2(t+1)}}{q ^{(k-t-1)(n-k)}} \\
			<	&\ \frac{1}{q^{\log_{q}{13s+2t-2} } } +\frac{s}{q^{2}\cdot q^{\log_{q}{13s} } } + \frac{s}{ q^{\log_{q}{13s} } }<\frac{1}{13s} +\frac{s}{4\cdot13s} +\frac{s}{13s} \\
			 < &\ \frac{1}{13} +\frac{1}{4\cdot13}+\frac{1}{13} =\frac{9}{52}<\frac{25}{26} .
			\end{aligned}
		\end{equation*}
		The required result follows.
		
		\medskip
		\noindent \textbf{Case 2.} $\tau_{t}(\mathcal{F})\ge k+1$ or $\tau_{t}(\mathcal{G})\ge k+1$. 
		\medskip
		
		W.l.o.g., we may assume that $\tau_{t}(\mathcal{G})\ge k+1$. From Lemma \ref{yinli2.4}, we obtain
			\begin{equation*}
			\begin{aligned}
			|\mf||\mg|\leq s{k\brack t}{n-t\brack k-t}\binom{2k-2t+2}{k-t+1} +s^{2} \binom{2k-2t+2}{k-t+1}.
			\end{aligned}
		\end{equation*}
	It follows from Lemma \ref{yinli6.2}(i) and (ii) that	
					\begin{equation*}
			\begin{aligned}
				|\mf||\mg|\leq &\ \frac{6}{7\cdot \frac{13}{7}q^{t}  } f_{1}(n,k,t,s,t+1){n-t\brack k-t}+\frac{6s}{7\cdot \frac{13}{7}q^{t} } f_{1}(n,k,t,s,t+1) \\
				=	&\	 \frac{6}{13q^{t} } f_{1}(n,k,t,s,t+1)\left ( {n-t\brack k-t}+s \right ) \\
				=	&\ \frac{6}{13q^{t} } \left ( {t+1\brack 1}{k-t+1\brack 1}{n-t-1\brack k-t-1}+s{t+1\brack 1} \right ) \left ( {n-t\brack k-t}+s \right ) .
			\end{aligned}
		\end{equation*}
	Then we have
		\begin{equation*}
	\begin{aligned}
		\frac{|\mf||\mg|}{{k-t+1\brack 1}{n-t-1\brack k-t-1}{n-t\brack k-t}} \le  &\ \frac{6}{13q^{t} } \left ( {t+1\brack 1}+\frac{s{t+1\brack 1}}{{k-t+1\brack 1}{n-t-1\brack k-t-1}}  \right ) \left ( 1+\frac{s}{{n-t\brack k-t}}  \right ) \\
		<	&\ \frac{6}{13q^{t} } \left ( 2q^{t}+\frac{2sq^{t}}{q^{(k-t)}\cdot q^{(k-t-1)(n-k)}}   \right ) \left ( 1+\frac{s}{q^{(k-t)(n-k)}}  \right )  \\
		<	&\ \left ( \frac{12}{13} +\frac{12s}{13q^{k-t}q^{(k-t-1)(n-k)}}  \right ) \left ( 1+\frac{s}{q^{2\log_{q}{13s} } }  \right )\\
		 < &\ \left ( \frac{12}{13} +\frac{12s}{13q^{2}q^{\log_{q}{13s} }}  \right ) \left ( 1+\frac{s}{ (13s)^{2}}  \right )\\
		< &\ \left ( \frac{12}{13} +\frac{3}{169}  \right ) \left ( 1+\frac{1}{ 169} \right ) <\frac{25}{26} .
	\end{aligned}
\end{equation*}
		This completes the proof.
		\end{proof}
		
	\begin{thm}\label{dingli2}
	Let $n$, $k$,  $t$ and $s$ be positive integers with $k\ge t+2$, and  $n\geq 3k+3t+1+\log_{q} 13s$. Suppose that $\mathcal{F},\mathcal{G}\subseteq  \qbinom{V}{k}$ are $s$-almost cross-$t$-intersecting which are not cross-$t$-intersecting. If $|\mathcal{F} ||\mathcal{G} |$ is maximum and $|\mathcal{F} |\le |\mathcal{G} |$, then there exist $X\subseteq  \qbinom{V}{k+1}$ and $E\subseteq  \qbinom{X}{t}$such that 
	
	$$\mf=\mh_{1}(V,E;k)\setminus \mathcal{A} ~~ and ~~ \mathcal{G}  =\mathcal{H} _{1} (V,E;k)\cup \mathcal{B}, $$
	where	$\ma$ is an $(q^{(k-t)(k-t+1)} \qbinom{n-k-1}{k-t}-s)$-subset of $\mathcal{H} _{2} (X,E;k)$ and $\mb$ is a $\min\left\{s, q^{(k-t+1)} \qbinom{t}{1}\right\}$-subset of ${X\brack k}\setminus\mh_{1}(X,E;k)$.
\end{thm}



		\begin{proof}[\bf Proof of Theorem \ref{dingli2}]	

		Let $n$, $k$, $t$ and $s$ be positive integers with $k\geq t+1$ and $n\geq 2k$. Suppose that $\mf,\mg \subseteq {V\b k}$ are maximal $s$-almost cross-$t$-intersecting families, but they are not cross-$t$-intersecting. Set $E= {\textstyle \sum_{T_{f}\in \mathcal{T} (\mathcal{F} )  }} T_{f} $ and $X= {\textstyle \sum_{T_{g}\in \mathcal{T} (\mathcal{G} )  }} T_{g} $. If $\left ( \tau _{t}\left ( \mathcal{F}  \right ), \tau _{t}\left ( \mathcal{G}  \right )  \right ) =(t,t+1) $, then the following claims hold.
	\begin{cl}\label{yinli3.3}
		
		\begin{enumerate}[\normalfont(i)]
			\item  $E\in {X\brack t}$. \label{yinli3.3(i)}
			\item For each $G\in \mathcal{G} \setminus \mathcal{G} _{E} $, we have $X\subseteq E+G\in {V\brack k+1}$.	\label{yinli3.3(ii)}
			\item For each $H\in \mathcal{H }_{1} \left ( V,E;k \right ) $ with $\mathcal{D}_{G} \left ( H,t \right ) \ne \emptyset  $, we have $H\cap X=E$.	\label{yinli3.3(iii)}
		\end{enumerate}
		
		 	\end{cl}
		 		\begin{proof}
		 	(i) For each $F\in \mf$, by $\tau_{t}(\mf)=t$, we obtain $T_{f}\subseteq F$ and then $E\subseteq F$. From Lemma \ref{yinli2.6}, we know that $\mathcal{T} (\mf)$ and $\mathcal{T} (\mg)$ are cross-$t$-intersecting. It follows from $\tau_{t}(\mf)=t$ that $E\subseteq T_{g}$ for any $T_{g}\in \mg$.
		 		
		 	Let $F\in \mf$ with $\mathcal{D}_{G} \left ( F,t \right ) \ne \emptyset $ and $T_{g}\in \tau_{t}(\mg)$. We have $E \subseteq F, E\subseteq  T_{g}$, then $E\subseteq  F\cap T_{g}$. It follows that $t\le \dim(E)\le \dim (F\cap T_{g})\le t$, which implies $\dim (E)=\dim (F\cap T_{g})=t$. Then $E=F\cap T_{g}\in {T_{g}\brack t}\subseteq {X\brack t} $.

(ii)	From  (\ref{yinli3.3(i)}), we know $\dim(E)=t$. For any $G \in \mathcal{G} \setminus \mathcal{G}_{E}$ and $T_{g}\in \mathcal{T} (\mg)$, since $E\subseteq T_{g}$ and $\dim(T_{g}\cap G )\ge t$, we have
			\begin{equation*}
				\begin{aligned}
					\dim(E\cap G)  = \dim(E\cap T_{g}  \cap G)=\dim(E)+\dim(T_{g}\cap G) -\dim\left \{ E +(T_{g}\cap G)\right \} \ge t-1,
				\end{aligned}
			\end{equation*}
			\begin{equation*}
				\begin{aligned}
					\dim( T_{g}  \cap (E+G))=&\  \dim(T_{g})+\dim(E+G)-\dim(T_{g}+G) \\
					\ge &\ t+1+k+1-\dim(T_{g}+G)\ge t+1,
				\end{aligned}
			\end{equation*}
			implying that $T_{g}\subseteq E+G$. Hence $X\subseteq E+G\in {V\brack k+1}$. 

(iii)	Let $H \in \mathcal{H}_{1} (V,E;k)$ with $\mathcal{D}_{G} \left ( H,t \right ) \ne \emptyset$ and $G\in \mathcal{D} _{\mathcal{G} }(H;t) $. Then $G\in \mathcal{G} \setminus \mathcal{G} _{E}$. By (\ref{yinli3.3(ii)}), we have $X\subseteq E+G\in {V\brack k+1}$, which implies $H\cap X  \subseteq  H\cap (E+G)$ and then $\dim\left ( H\cap X \right ) \le \dim\left ( H\cap (E+G) \right ) $. Then we have
\begin{equation*}
	\begin{aligned}
		\dim(H\cap X)-1\le &\ \dim ( H\cap(E+G) )  +\dim(G)-\dim(E+G)\\
		\le &\ \dim ( H\cap G\cap(E+G) )\\
		=&\ \dim(H\cap G) \le t-1,
	\end{aligned}
\end{equation*}
which implies that $\dim(H\cap X)\le t$. This together with $E\subseteq H\cap X$ and $\dim(E)=t$ yields the desired results.
	\end{proof}
	\begin{cl}\label{yinli3.4}
			Let $n$, $k$, $t$ and $s$ be positive integers with $k\geq t+2$ and $n\geq  2k+3t+1+\log_{q}{13s}$. Suppose that $\mathcal{F}, \mathcal{G}\subseteq {V\brack k}$ are maximal $s$-almost cross- $t$-intersecting families, but they are not cross-$t$-intersecting. If $(\tau_{t}(\mf),\tau_{t}(\mg))=(t,t+1)$, then $\left | \mathcal{G}  \right | \le {n-t\brack k-t}+q^{\dim(X)-t+1} {k-t\brack 1}{t\brack 1}{n-\dim(X)\brack k-\dim(X)}+s $, where $X=\sum_{T_{g}\in \mathcal{T} (\mathcal{G} )  }T_{g}$.
	\end{cl}
	\begin{proof}
	Let $A\in \mf$ and $B\in \mg$ with $\dim(A\cap B)<t$. Then $B\in \mathcal{G} \setminus \mathcal{G}_{E} $ due to $\mathcal{D} _{\mathcal{F} } (B;t)\ne \emptyset $, and $E\in {X\brack t}$ due to Claim \ref{yinli3.3}(\ref{yinli3.3(i)}).
	
	If $\mathcal{G}= \mathcal{G}_{E} \cup \mathcal{D} _{\mathcal{G}} (A;t)$, then $\left | \mathcal{G} \right | \le {n-t\brack k-t}+s$, the required result holds. Next we assume that $\mathcal{G}\setminus \left ( \mathcal{G}_{E} \cup \mathcal{D} _{\mathcal{G}} (A;t) \right ) \ne \emptyset $.
	
	For each $G\in \mathcal{G}\setminus \left ( \mathcal{G}_{E} \cup \mathcal{D} _{\mathcal{G}} (A;t) \right ) $, by $E\subseteq X$ and Claim \ref{yinli3.3}(\ref{yinli3.3(ii)}), we get $E+G\subseteq X+G\subseteq E+G$ and $\dim(E+G)=k+1$, implying that 
	$$ \dim(X+G)=k+1= \dim(X)+\dim(G)-\dim(X\cap G)=\dim(X)+k-\dim(X\cap G),$$ then $\dim(X\cap G)=\dim(X)-1$ and $\dim(G\cap E)=t-1$. It follows from Claim \ref{yinli3.3}(\ref{yinli3.3(iii)})  that 
		\begin{equation*}
			\begin{aligned}
k\ge \dim ( G\cap (A+X) )\ge &\ \dim (  (G\cap A)+(G\cap X) )\\                                          =&\ \dim(G\cap A)+\dim(G\cap X)-\dim\left ( (G\cap A)\cap (G\cap X)  \right )  \\                 =&\ \dim(G\cap A)
+\dim(G\cap X)-\dim(G\cap E) \\               \ge &\ t+\dim(X)-1-t+1=\dim(X),
			\end{aligned}
		\end{equation*}
		which implies $\dim(X)\le k$. Therefore, we have 
		$$\mathcal{G} \setminus \left ( \mathcal{G} _{E} \cup \mathcal{D} _{\mathcal{G} }(A;t)  \right ) \subseteq \bigcup_{H\subseteq {A+X\brack \dim(X)},~ \dim(H\cap X)=\dim(X)-1,~ E\nsubseteq H}\mathcal{G}_{H}  .$$
		Note that
		$$\left | \left \{ H\subseteq {A+X\brack \dim(X)}: \dim(H\cap X)=\dim(X)-1,E\nsubseteq H \right \}  \right | =q^{\dim(X)-t+1}{k-t \brack 1}{t\brack 1} .$$
		We further conclude $\left | \mathcal{G} \setminus \left ( \mathcal{G} _{E} \cup \mathcal{D} _{\mathcal{G} }(A;t)  \right ) \right | \le q^{\dim(X)-t+1}{k-t \brack 1}{t\brack 1}{n-\dim(X)\brack k-\dim(X)}$. This yields
		$$\left | \mathcal{G}  \right | \le q^{\dim(X)-t+1}{k-t \brack 1}{t\brack 1}{n-\dim(X)\brack k-\dim(X)}+{n-t\brack k-t}+s$$
		due to $\mathcal{G}=\left ( \mathcal{G} \setminus \left ( \mathcal{G} _{E} \cup \mathcal{D} _{\mathcal{G} }(A;t)  \right ) \right )\cup\mathcal{G} _{E}\cup \mathcal{D} _{\mathcal{G} }(A;t) $.
	\end{proof}

 Since $|\mathcal{F} | |\mathcal{G} |$ takes the maximum value, by Example \ref{lizi}, we know 
					\begin{align}\label{(3.4)}
				|\mf||\mg|\ge g_{1}(n,k,t,s).
			\end{align}
		Recall that $\left ( \tau _{t}(\mathcal{F} )  ,\tau _{t}(\mathcal{G} )  \right )\in \left \{ (t,t+1), (t+1,t)\right \}  $.
		
			\begin{cl}\label{duanyan3.5}
			$\left ( \tau _{t}(\mathcal{F} )  ,\tau _{t}(\mathcal{G} )  \right )=(t,t+1)$.
		\end{cl}
			\noindent \textit{Proof.} 
			Suppose for condition that $\left ( \tau _{t}(\mathcal{F} )  ,\tau _{t}(\mathcal{G} )  \right )=(t+1,t)$. By Lemma \ref{yinli2.2}, we know $|\mg|\le {k-t+1\brack 1} {n-t-1\brack k-t-1} +s$. This together with $|\mf|\le |\mg|$ and Lemma \ref{yinli6.3} (ii) yields
			$$\left | \mathcal{F}  \right |  \left | \mathcal{G}  \right | \le \left ( {k-t+1\brack 1} {n-t-1\brack k-t-1} +s \right ) ^{2}<g_{1}(n,k,t,s), $$
			a contradiction to (\ref{(3.4)}). Therefore, we have $\left ( \tau _{t}(\mathcal{F} )  ,\tau _{t}(\mathcal{G} )  \right )=(t,t+1)$.
			$\hfill\square$


	\begin{cl}\label{duanyan3.6}
$X=E+G\in {V\brack k+1} $ for any $G\in \mathcal{G} \setminus \mathcal{G}_{E} $.
	\end{cl}
	\noindent \textit{Proof.} 
From Claim \ref{yinli3.3}(\ref{yinli3.3(ii)}) , we obtain $X\subseteq E+G\in {V\brack k+1} $ for any $G\in \mathcal{G} \setminus \mathcal{G}_{E} $.

We first show that $E+B=E+G$ for any  $G\in \mathcal{G} \setminus \mathcal{G}_{E} $. Suppose for contradiction that there exist $C\in G\in \mathcal{G} \setminus \mathcal{G}_{E}$ such that $E+B\ne E+C$. Let $M=(E+B)\cap (E+C)$. By $X\subseteq M \subsetneq E+B\in {V\brack k+1}$, we know $t+1\le \dim(M)\le k$. Set 
$$\mathcal{H} =\left \{ (H_{1} ,H_{2} )\in {E+B\brack t+1}\times {E+C\brack t+1}:~E\subseteq H_{1} \nsubseteq M,~E\subseteq H_{2} \nsubseteq M \right \} .$$
For each $F\in \mathcal{F} \setminus \left ( \mathcal{D} _{\mathcal{F} } (B;t)\cup \mathcal{D} _{\mathcal{F} } (C;t) \right )$, since $E\subseteq F$, $\dim(E\cap B)=t-1$, $\dim(E\cap C)=t-1$ and $\dim(F+E+B)=\dim(F+B)= \dim(F)+\dim(B)-\dim(F\cap B)\le 2k-t$, we have 

\begin{equation*}
	\begin{aligned}
		\dim(F\cap (E+B))= &\dim(F)+\dim (E+B)-\dim(F+E+B)\geq t+1
	\end{aligned}
\end{equation*} 
and $\dim(F\cap (E+C))\ge t+1$ similarly. Therefore we get
$$\mathcal{F} \subseteq \left \{ F\in {V\brack k}:E\subseteq F,\dim(F\cap M)\ge t+1 \right \} \cup \left \{ \bigcup_{(H_{1},H_{2} )\in \mathcal{H} }\mathcal{F}_{H_{1}+H_{2}} \right \} \cup \mathcal{D} _{\mathcal{F} }(B;t)\cup \mathcal{D} _{\mathcal{F} }(C;t) .$$
Note that $\dim(H_{1}+ H_{2}) = \dim(H_{1}) +\dim( H_{2}) -\dim(H_{1}\cap H_{2}) =2t+2-t=t+2$ for any $(H_{1},H_{2} )\in \mathcal{H}$. Then
\begin{equation*}
	\begin{aligned}
		|\mf|\le &\ {\dim(M)-t\brack 1} {n-t-1\brack k-t-1}+q^{2(\dim(M)-t)} {k+1-\dim(M)\brack 1}^{2} {n-t-2\brack k-t-2}+2s\\
		=&\ f_{2}(n,k,t,s,\dim(M)) .
	\end{aligned}
\end{equation*} 
This together with Lemma \ref{yinli6.4} (i) and $\dim(M)\le k$ yields $\left | \mathcal{F}  \right | \le f_{2} (n,k,t,s,k)$. From Lemma \ref{yinli3.4}, \ref{yinli6.5} and $\dim(X)\ge t+1$, we obtain
$$\left | \mathcal{F}  \right | \left | \mathcal{G}  \right |\le f_{2} (n,k,t,s,k)\left ( {n-t\brack k-t}+q^{2}{k-t \brack 1}{t\brack 1}{n-t-1\brack k-t-1}+s  \right ) .$$
It follows from Lemma \ref{yinli6.4} (ii) that $\left | \mathcal{F}  \right | \left | \mathcal{G}  \right |< g_{1}(n,k,t,s) $, a contradiction to (\ref{(3.4)}). Hence $E+B=E+G$ for any $G\in \mathcal{G} \setminus \mathcal{G}_{E}$.

Let $G\in \mg$ and $T\in {E+B\brack t+1}$ with $E\in T$. If $G\in \mathcal{G}_{E}$, then $\dim(T\cap G)\ge t$. If $G\in \mathcal{G} \setminus \mathcal{G}_{E}$, then by $\dim(E+G)=k+1$ and $E+G=E+B$, we have 
\begin{equation*}
	\begin{aligned}
	\dim(T\cap G)=&\ \dim(T)+\dim(G)-\dim(T+ G)\\
	             =&\ \dim(T)+\dim(G)-\dim(E+ G) =t.
	\end{aligned}
\end{equation*} 
Therefore, we have $T\in \mathcal{T} (\mathcal{G} )$. We further conclude $\left \{ T\in {E+B\brack t+1}:E\subseteq T \right \} \subseteq \mathcal{T}  (\mathcal{G} ) $, which implies $E+B\subseteq X$. This together with $X\subseteq E+B\in {V\brack k+1} $ yields $X=E+B=E+G$ for any $G\in \mathcal{G} \setminus \mathcal{G}_{E}$.
$\hfill\square$

	Since $\tau _{t} (\mathcal{F} )=t$, we have $\mf \subseteq \mathcal{H} _{1} (V,W;k) $. By Claim \ref{duanyan3.6}, we know $X=E+B\in {V\brack k+1}$, which implies $\mathcal{D} _{\mathcal{H}_{1}(V,E;k)  } (B;t)=\left \{ F\in {V\brack k}:F\cap X=E \right \} =\mathcal{H}_{2}(X,E;k)$. It follows from $n>2k+3t+1+\log_{q}{13s}$ and $k\ge t+2$ that
	\begin{equation*}
		\begin{aligned}
			\left | \mathcal{D} _{\mathcal{H}_{1}(V,E;k)  } (B;t) \right | = &\ q^{(k-t+1)(k-t)}{n-k-1\brack k-t}\ge q^{6}q^{(k-t)(n-2k+t-1)}\\
			\ge &\ q^{6}(13s)^{2}>s .
		\end{aligned}
	\end{equation*}  
Hence $\mf \subsetneq  \mathcal{H}_{1}(V,E;k) $. To determine $\mf$, it is sufficient to describe $ \mathcal{H}_{1}(V,E;k)\setminus \mathcal{F}$.

Pick $H \in \mathcal{H}_{1}(V,E;k)\setminus \mathcal{F}$. By the maximality of $\mf$ and $\mg$, we know $\mathcal{D}_{\mathcal{G} }(H;t) \ne \emptyset $. From Claim \ref{yinli3.3(iii)} , we obtain $H\in  \mathcal{H}_{2}(X,E;k) $, which implies 
\begin{align}\label{(3.5)}
\mathcal{H}_{1}(V,E;k)\setminus \mathcal{F}\subseteq  \mathcal{H}_{2}(X,E;k) .
\end{align}
It follows that 
$$\mathcal{H}_{2}(X,E;k)\setminus \mathcal{F}\subseteq  \mathcal{H}_{1}(V,E;k)\setminus \mathcal{F}\subseteq \mathcal{H}_{2}(X,E;k)\setminus \mathcal{F} . $$
Therefore, we have $\mathcal{H}_{1}(V,E;k)\setminus \mathcal{F}\subseteq \mathcal{H}_{2}(X,E;k)\setminus \mathcal{F}$. This combining with $\mathcal{D}_{\mathcal{F} } (B;t)= \mathcal{H}_{2}(X,E;k)\cap  \mathcal{F}$ yields
\begin{align}\label{(3.6)}
\left | \mathcal{H}_{1}(V,E;k)\setminus \mathcal{F}  \right | =\left | \mathcal{H}_{2}(X,E;k) \right |-\left | \mathcal{H}_{2}(X,E;k)\cap \mathcal{F}  \right |\ge q^{(k-t+1)(k-t)} {n-k-1\brack k-t}-s.
\end{align}

Choose $G\in \mathcal{G}\setminus \mathcal{G}_{E}$. By Claim \ref{duanyan3.6}, we have $G\subseteq E+G=X $. Hence 
\begin{align}\label{(3.7)}
	\mathcal{G} \setminus \mathcal{G}_{E}\subseteq {X\brack k}\setminus \mathcal{H}_{1} (X,E;k).
\end{align}
	From Claim \ref{yinli3.3}(\ref{yinli3.3(iii)}) , we know $A\cap X=E$, which implies $\dim(A\cap G)=\dim(A\cap X \cap G)=\dim(E\cap G)<t$. It follows that
	\begin{align}\label{(3.8)}
		\left | \mathcal{G}\setminus \mathcal{G}_{E} \right |\le \min \left \{ \left | {X\brack k}\setminus  \mathcal{H}_{1}(X,E;k) \right |,\left | \mathcal{D} _{\mathcal{G} }(A;t) \right |   \right \}\le \min \left \{ q^{k+1-t}{t\brack 1},s  \right \}   .
	\end{align}
	Now, according to (\ref{(3.4)}), (\ref{(3.6)}) and (\ref{(3.8)}), we have 
		\begin{equation*}
		\begin{aligned}
			g_{1}(n,k,t,s)\ge &\ \left ( \left | \mathcal{H}_{1} (V,E;k) \right |- \left | \mathcal{H}_{1} (V,E;k)\setminus \mathcal{F}  \right |     \right )\left ( \left | \mathcal{G}_{E}  \right |+\left | \mathcal{G} \setminus \mathcal{G}_{E}  \right |    \right ) =|\mf||\mg| \ge  g_{1}(n,k,t,s).
		\end{aligned}
	\end{equation*}  
	This together with (\ref{(3.5)}) and (\ref{(3.7)}) implies $\mathcal{H} _{1}(V,E;k)\setminus \mathcal{F} $ is a $(q^{(k-t)(k-t+1)} \qbinom{n-k-1}{k-t}-s)$-subset of $\mathcal{H} _{2} (X,E;k)$, $\mathcal{G} _{E}= \mathcal{H} _{1} (V,E;k)$ and $\mathcal{G}\setminus \mathcal{G}_{E}$ is a $\min\left\{s, q^{(k-t+1)} \qbinom{t}{1}\right\}$-subset of ${X\brack k}\setminus\mh_{1}(X,E;k)$. Then the desired result holds.
		\end{proof}
Now we are ready to prove Theorem \ref{dingli5.2}.	
		\begin{proof}
		Since $|\mathcal{F} ||\mathcal{G} |$ take the maximum value, by (\ref{lizi}) and (\ref{buchonglizi}), we have
		\begin{equation}\label{5.2}
			|\mathcal{F} ||\mathcal{G} |\ge \max\left \{  g_{1} (n,k,t,s),g_{2} (n,k,t) \right \} .
		\end{equation}
		We divide our proof into the following two cases.
		
		\medskip
		\noindent{{\bf Case 1.} $k\le 2t$.}
		\medskip
		
		By (\ref{5.2}) and Lemma \ref{6.5} (ii), we have $|\mathcal{F} ||\mathcal{G} |\ge g_{3} (n,k,t) >  g_{1} (n,k,t,s)$. It follows from Theorem \ref{dingli2} that $\mathcal{F} $ and $\mathcal{G} $ are cross-$t$-intersecting.
		
		Since $|\mathcal{F} ||\mathcal{G} |$ take the maximum value, by \cite[Theorem 1.2]{交叉t相交向量4} , there exists $L\subseteq  \qbinom{V}{t+1}$ such that $\mathcal{F} =\mathcal{H}_{1} (V,L;k)$ and $\mathcal{G}  =\mathcal{M}(L;k,t) $, or $\mathcal{F} =\mathcal{M} (L;k,t)$ and $\mathcal{G}  = \mathcal{H}_{1} (V,L;k)$. This together with $|\mathcal{F} |\le |\mathcal{G} |$ yields $\mathcal{F} =\mathcal{M} (L;k,t)$ and $\mathcal{G}  = \mathcal{H}_{1} (V,L;k)$. Then (i) holds.
		
		\medskip
		\noindent{{\bf Case 2.} $k\ge 2t+1$.}
		\medskip
		
		By (\ref{5.2}) and Lemma \ref{6.5}(i), we have 
		$$|\mathcal{F} ||\mathcal{G} |\ge g_{1} (n,k,t,s)>\left ( {n-t\brack k-t}-q^{(k-t)(k+1-t)}{n-k-1\brack k-t}  \right ) \left ( {n-t\brack k-t}+ q^{(k-t+1)} \qbinom{t}{1} \right ). $$
		According to \cite[Theorem 1.2]{交叉t相交向量4}, we know that $\mathcal{F}$ and $\mathcal{G}$ are not cross-$t$-intersecting.
		
		Since $|\mathcal{F} ||\mathcal{G} |$ taking maximum value and $|\mathcal{F} |\le |\mathcal{G} |$, by Theorem \ref{dingli2}, we conclude that (ii) follows.
	\end{proof}

	\medskip
	\noindent{\bf Acknowledgment.}	
	T. Yao is supported by Natural Science Foundation of Henan (252300420899).

\end{document}